\documentclass[11pt, reqno]{amsart}
\usepackage{amsmath,amsopn,amssymb,amsthm,multicol}
\usepackage[cp1251]{inputenc}
\usepackage[english]{babel}
\usepackage[breaklinks=true,colorlinks=true,linkcolor=blue,citecolor=blue,urlcolor=blue]{hyperref}
\usepackage[dvipdf]{graphicx}

\renewenvironment{proof}[1][Proof]{\textbf{#1.} }
{\ \rule{0.5em}{0.5em}}

\DeclareMathOperator{\Ad}{Ad}
\DeclareMathOperator{\diag}{diag}
\DeclareMathOperator{\Ric}{Ric}

\renewcommand{\arraystretch}{1.5}

\newtheorem{theorem}{Theorem}
\newtheorem{pred}{Proposition}
\newtheorem{lem}{Lemma}
\newtheorem{cor}{Corollary}

\newtheorem{remark}{Remark}
\newtheorem{ques}{Question}

\begin{document}

\title[Invariant Einstein metrics on generalized Wallach spaces]
{Invariant Einstein metrics on generalized Wallach spaces}

\author{Zhiqi~Chen}
\address{Zhiqi~Chen\newline
School of Mathematical Sciences and LPMC, Nankai University, \newline
Tianjin 300071, China}
\email{chenzhiqi@nankai.edu.cn}

\author{Yuri\u{i}~Nikonorov}
\address{Yu.\,G. Nikonorov \newline
Southern Mathematical Institute of Vladikavkaz Scientific Centre \newline
of the Russian Academy of Sciences, Vladikavkaz, Markus st. 22, \newline
362027, Russia}
\email{nikonorov2006@mail.ru}

\thanks{The project was supported in part by Grant 1452/GF4 of Ministry of Education and Sciences of the Republic of Kazakhstan
for 2015-2017 (Agreement N 299, February 12, 2015).}

\begin{abstract}
Invariant Einstein metrics on generalized Wallach spaces have been classified except $SO(k+l+m)/SO(k)\times SO(l)\times SO(m)$. In this paper,
we give a survey on the study of invariant Einstein metrics on generalized
Wallach spaces, and prove that there are infinitely many spaces of the type $SO(k+l+m)/SO(k)\times SO(l)\times SO(m)$
admitting exactly two, three, or four invariant Einstein metrics up to a homothety.

\vspace{2mm} \noindent Key word and phrases:
compact homogeneous space, generalized Wallach space, symmetric space, homogeneous Riemannian metric, Einstein metric, Ricci flow.

\vspace{2mm}

\end{abstract}

\maketitle

\section{Introduction}
This paper is devoted to the study of invariant Einstein metrics on generalized Wallach spaces,
a remarkable class of compact homogeneous spaces which were introduced in \cite{Nikonorov2}.
They werecalled three-locally-symmetric spaces there, here we prefer to use the term generalized Wallach spaces as \cite{Nikonorov1}.
\smallskip

Let $G$ a connected compact semisimple Lie group $G$ and $H$ its compact subgroup. Denote by $\mathfrak{g}$ and $\mathfrak{h}$ Lie algebras of $G$ and $H$ respectively. Assume that the compact homogeneous space $G/H$ is almost effective, i.~e. there is no non-trivial ideal of the Lie algebra $\mathfrak{g}$ in $\mathfrak{h} \subset\mathfrak{g}$. Naturally, the Killing form $B=B(\boldsymbol{\cdot}\,,\boldsymbol{\cdot})$ of $\mathfrak{g}$ is negative definite.
It follows that $\langle\boldsymbol{\cdot}\,,\boldsymbol{\cdot}\rangle:=-B(\boldsymbol{\cdot}\,,\boldsymbol{\cdot})$ is a positive definite inner product
on $\mathfrak{g}$. Let $\mathfrak{p}$ be the $\langle\boldsymbol{\cdot}\,,\boldsymbol{\cdot}\rangle$-orthogonal complement to $\mathfrak{h}$ in $\mathfrak{g}$.
Therefore $\mathfrak{p}$ is $\operatorname{Ad}(H)$-invariant (and $\operatorname{ad}(\mathfrak{h})$-invariant, in particular).
The module $\mathfrak{p}$ is naturally identified with the tangent space to $G/H$ at the point $eH$.
Every $G$-invariant Riemannian metric on~$G/H$ generates an~$\operatorname{Ad}(H)$-invariant
inner product on~$\mathfrak{p}$ and vice versa. Hence, it is natural to
identify invariant Riemannian metrics on~$G/H$ with $\operatorname{Ad}(H)$-invariant inner
products on~$\mathfrak{p}$. Note that the Riemannian metric generated by the~inner product
$\langle\boldsymbol{\cdot}\,,\boldsymbol{\cdot}\rangle\bigr\vert_{\mathfrak{p}}$ is called
{\it Killing} or {\it standard}. A compact homogeneous space~$G/H$ is called a {\it generalized Wallach space} if the module $\mathfrak{p}$
is decomposed as a direct sum of three
$\operatorname{Ad}(H)$-invariant irreducible modules pairwise orthogonal
with respect to~$\langle\boldsymbol{\cdot}\,,\boldsymbol{\cdot}\rangle$, i.~e.
$\mathfrak{p}=\mathfrak{p}_1\oplus \mathfrak{p}_2\oplus \mathfrak{p}_3$,
such that
$[\mathfrak{p}_i,\mathfrak{p}_i]\subset \mathfrak{h}$ for $i\in\{1,2,3\}$.
\smallskip

There are many examples of generalized Wallach spaces, such as the manifolds of complete flags in the complex, quaternionic, and Cayley projective planes:
$SU(3)/T_{\max}$, $Sp(3)/Sp(1)\times Sp(1)\times Sp(1)$, and $F_4/Spin(8)$, which are known as {\it Wallach spaces};
the various K\"ahler $C$-spaces $SU(n_1+n_2+n_3)\big/S\big(U(n_1)\times U(n_2)\times U(n_3)\big)$, $
SO(2n)/U(1)\times U(n-1)$, and $E_6/U(1)\times U(1)\times Spin(8)$.
The classification of generalized Wallach spaces  $G/H$ is obtained in \cite{Nikonorov4}, and the part for simple Lie groups $G$ is also given in \cite{CKL}.

\begin{theorem}[\cite{Nikonorov4}]\label{classgws}
A simply connected generalized Wallach space $G/H$ is exactly one of the following cases:
\begin{enumerate}
 \item $G/H$ is a direct product of three irreducible symmetric spaces of compact type;
 \item The group $G$ is simple and the pair $(\mathfrak{g}, \mathfrak{h})$ is one of the pairs in Table 1;
 \item $G=F\times F\times F \times F$ and $H=\operatorname{diag}(F)\subset G$ for some connected simply connected compact simple Lie group $F$,
with the following description on the Lie algebra level:
$$
(\mathfrak{g}, \mathfrak{h})= \bigl(\mathfrak{f}\oplus \mathfrak{f}\oplus\mathfrak{f}\oplus\mathfrak{f},\,
\operatorname{diag}(\mathfrak{f})=\{(X,X,X,X)\,|\, X\in \mathfrak{f}\}\bigr),
$$
where $\mathfrak{f}$ is the Lie algebra of $F$, and {\rm (}up to permutation{\rm)}
$\mathfrak{p}_1=\{(X,X,-X,-X)\,|\, X\in \mathfrak{f}\}$, \linebreak
$\mathfrak{p}_2\!=\!\{(X,-X,X,-X)\,|\, X\in \mathfrak{f}\}$,\!
$\mathfrak{p}_3\!=\!\{(X,-X,-X,X)\,|\, X\in \mathfrak{f}\}$.
\end{enumerate}
\end{theorem}

{\small
\renewcommand{\arraystretch}{1.4}
\begin{table}[t]
{\bf Table 1.} The pairs $(\mathfrak{g},\mathfrak{h})$ corresponded to generalized Wallach spaces $G/H$ with $G$ simple.
\begin{center}
\begin{tabular}{|c|c|c|c|c|c|c|c|c|}
\hline
N& \ $\mathfrak{g}$& $\mathfrak{h}$ & $d_1$ &  $d_2$ & $d_3$  & $a_1$ &  $a_2$ & $a_3$ \\
\hline\hline
1&\!\!$so(k\!+\!l\!+\!m)$\!\!\!&\!\!$so(k)\!\oplus\! so(l) \!\oplus \!so(m)$\!\!& $kl$ & $km$ & $lm$ &
\!\!$\frac{m}{2(k+l+m-2)}$\!\!&\!\!$\frac{l}{2(k+l+m-2)}$\!\!&\!\!$\frac{k}{2(k+l+m-2)}$\!\!\\  \hline
2&\!\!$su(k\!+\!l\!+\!m)$\!\!\!&$\!\!s(u(k)\!\oplus \!u(l) \!\oplus\! u(m))$\!\!& $2kl$ & $2km$ & $2lm$ &
\!\!$\frac{m}{2(k+l+m)}$\!\!&\!\!$\frac{l}{2(k+l+m)}$\!\!&\!\!$\frac{k}{2(k+l+m)}$\!\!\\  \hline
3&\!\!$sp(k\!+\!l\!+\!m)$\!\!\!&\!\!$sp(k)\!\oplus \!sp(l) \!\oplus \!sp(m)$\!\!& $4kl$ & $4km$ & $4lm$ &
\!\!$\frac{m}{2(k+l+m+1)}$\!\!&\!\!$\frac{l}{2(k+l+m+1)}$\!\!&\!\!$\frac{k}{2(k+l+m+1)}$\!\!\\  \hline
4&$su(2l)$,\,$l\geq 2$ &$u(l)$&\!\!$l(l-1\!)$\!\!&\!\!$l(l+1\!)$\!\!&\!\!$l^2-1$\!\!&  $\frac{l+1}{4l}$ &
$\frac{l-1}{4l}$ &  $\frac{1}{4}$ \\  \hline
5&$so(2l)$,\,$l\geq 4$ &$u(1)\oplus u(l-1\!)$ &\!\!$2(l-1\!)$\!\! &\!\!$2(l-1\!)$\!\!&\!\!$(l\!-\!1)\!(l\!-\!2)$\!\!&  $\frac{l-2}{4(l-1)}$ &
$\frac{l-2}{4(l-1)}$ &  $\frac{1}{2(l-1)}$ \\  \hline
6&$e_6$ &$su(4)\oplus 2sp(1)\oplus \mathbb{R}$   &$16$ & $16$ & $24$ &  $\frac{1}{4}$ &  $\frac{1}{4}$ &  $\frac{1}{6}$ \\  \hline
7&$e_6$ &$so(8)\oplus \mathbb{R}^2$ & $16$ & $16$ & $16$ &  $\frac{1}{6}$ &  $\frac{1}{6}$ &  $\frac{1}{6}$ \\  \hline
8&$e_6$ &$sp(3)\oplus sp(1)$ & $14$ & $28$ & $12$ &  $\frac{1}{4}$ &  $\frac{1}{8}$ &  $\frac{7}{24}$ \\  \hline

9&$e_7$ &$so(8)\oplus 3sp(1)$ & $32$ & $32$ & $32$ &  $\frac{2}{9}$ &  $\frac{2}{9}$ &  $\frac{2}{9}$ \\  \hline
10&$e_7$ &$su(6)\oplus sp(1)\oplus \mathbb{R}$ & $30$ & $40$ & $24$ &  $\frac{2}{9}$ &  $\frac{1}{6}$ &  $\frac{5}{18}$ \\  \hline
11&$e_7$ &$so(8)$ & $35$ & $35$ & $35$ &  $\frac{5}{18}$ &  $\frac{5}{18}$ &  $\frac{5}{18}$ \\  \hline

12&$e_8$ &$so(12)\oplus 2sp(1)$ & $64$ & $64$ & $48$ &  $\frac{1}{5}$ &  $\frac{1}{5}$ &  $\frac{4}{15}$ \\  \hline
13&$e_8$ &$so(8)\oplus so(8)$ & $64$ & $64$ & $64$ &  $\frac{4}{15}$ &  $\frac{4}{15}$ &  $\frac{4}{15}$ \\  \hline

14&$f_4$ &$so(5)\oplus 2sp(1)$ & $8$ & $8$ & $20$ &  $\frac{5}{18}$ &  $\frac{5}{18}$ &  $\frac{1}{9}$ \\  \hline
15&$f_4$ &$so(8)$ & $8$ & $8$ & $8$ &  $\frac{1}{9}$ &  $\frac{1}{9}$ &  $\frac{1}{9}$ \\  \hline

\end{tabular}
\end{center}
\end{table}
\renewcommand{\arraystretch}{1}
}

The definitions and the properties of the numbers $d_i$ and $a_i$, $i=1,2,3$, in Table 1 are considered in the next section.
\smallskip

A Riemannian metric is Einstein if the Ricci curvature is a constant multiple of the metric.
Various results on Einstein manifolds could be found in the book \cite{Bes} of A.L.~Besse and in more recent surveys \cite{Nikonorov1, Wang1, Wang2}.

There are a lot of studies on invariant Einstein metrics on generalized Wallach spaces. The~invariant Einstein metrics on~$SU(3)/T_{\max}$
were classified in~\cite{Dat} and,
on $Sp(3)/Sp(1)\times Sp(1)\times Sp(1)$ and $F_4/Spin(8)$
in~\cite{Rod}. In~each of these~cases, there exist exactly
four invariant Einstein metrics (up to proportionality). The~invariant Einstein metrics on
$SU(n_1+n_2+n_3)\big/S\big(U(n_1)\times U(n_2)\times U(n_3)\big)$, $SO(2n)/U(1)\times U(n-1)$, and $E_6/U(1)\times U(1)\times Spin(8)$ were
classified in~\cite{Kim}. Each of these spaces admits four
invariant Einstein metrics (up to scalar), one of which is K\"ahler for an~appropriate
complex structure on~$G/H$. Another approach to
$SU(n_1+n_2+n_3)\big/S\big(U(n_1)\times U(n_2)\times U(n_3)\big)$ was used
in~\cite{Ar}. As a~generalized Wallach space, the~Lie group~$SU(2)$ \big($H=\{e\}$\big) admits only one left-invariant Einstein metric which is
a~metric of constant curvature~(see e.~g. \cite{Bes}). Recall that $SU(2)$ is locally isomorphic to  $SO(3)=SO(3)/SO(1)\times SO(1)\times SO(1)$.
\smallskip

In~\cite{Nikonorov2}, it was shown that every generalized Wallach space admits at least one invariant Einstein metric. Later in \cite{Lomshakov2},
a detailed study of invariant Einstein metrics was developed for all generalized Wallach spaces. In particular, it is proved that there are at
most four Einstein metrics (up to homothety) for every such space. In~\cite{CNN}, the authors classified all invariant Einstein metrics on
Ledger--Obata spaces $F^4/\diag(F)$, hence, Einstein metrics of generalized Wallach spaces in the item 3) of Theorem~\ref{classgws}. By the results from \cite{CKL}, \cite{Lomshakov2} and \cite{Nikonorov1}, all invariant Einstein metrics on generalized Wallach spaces in the item 2) of Theorem~\ref{classgws} except $SO(k+l+m)/SO(k)\times SO(l)\times SO(m)$ were classified.
\smallskip

In the recent papers \cite{AANS1,AANS2}, generalized Wallach spaces were studied from the point of view of the Ricci flow. Recall that singular points
of the normalized Ricci flow for a homogeneous space $G/H$ are exactly
invariant Einstein metrics on $G/H$. It allows us to do a further study about invariant Einstein metrics on $SO(k+l+m)/SO(k)\times SO(l)\times SO(m)$.
More or less complete general picture we get from Theorem~7 in \cite{AANS2} (see also Theorem 6 in \cite{AANS1}) and \cite{CKL}.
Other helpful facts were obtained in \cite[Section 6]{Nikonorov4}.

\smallskip

Our main results are the following two theorems.
\smallskip

\begin{theorem}\label{ems} Suppose that $k\geq l \geq m \geq 1$ and $l\geq 2$. Then
the number of invariant Einstein metrics on the space $G/H=SO(k+l+m)/SO(k)\times SO(l)\times SO(m)$ is
$4$ for $m> \sqrt{2k+2l-4}$  and  $2$ for $m<\sqrt{k+l}$,  up to a homothety.
\end{theorem}

\begin{theorem}\label{emsn} Let $q$ be any number from the set $\{2,3,4\}$. Then there are infinitely many homogeneous spaces
$SO(k+l+m)/SO(k)\times SO(l)\times SO(m)$
that admit exactly $q$ invariant Einstein metrics up to a homothety.
\end{theorem}

The proofs of these theorems are based on a deep study of some special subsets in $(0,1/2)^3 \subset \mathbb{R}^3$
related to the normalized Ricci flow on the spaces $SO(k+l+m)/SO(k)\times SO(l)\times SO(m)$ and are obtained in the last section of this paper.
\smallskip

\section{Description of Einstein metrics on generalized Wallach spaces}
Every generalized Wallach space admits a $3$-parameter family of invariant Riemannian metrics
determined by inner products
$$(\boldsymbol{\cdot}\,,\boldsymbol{\cdot})=\left.x_1\,\langle\boldsymbol{\cdot}\,,\boldsymbol{\cdot}\rangle\right|_{\mathfrak{p}_1}+
\left.x_2\,\langle\boldsymbol{\cdot}\,,\boldsymbol{\cdot}\rangle\right|_{\mathfrak{p}_2}+
\left.x_3\,\langle\boldsymbol{\cdot}\,,\boldsymbol{\cdot}\rangle\right|_{\mathfrak{p}_3},$$
where $x_1$, $x_2$, $x_3$ are positive real numbers. Denote by $d_i$ the~dimension of~$\mathfrak{p}_i$. Let~$\big\{e^j_i\big\}$ be
an~orthonormal basis in~$\mathfrak{p}_i$ with respect to
$\langle\boldsymbol{\cdot}\,,\boldsymbol{\cdot}\rangle$ for $i=1,2,3$ and $1\le j\le d_i=\dim(\mathfrak{p}_i)$. Define
$$[ijk]=\sum_{\alpha,\beta,\gamma}
\left\langle\big[e_i^\alpha,e_j^\beta \big],e_k^\gamma \right\rangle^2,$$
where $\alpha$, $\beta$, and $\gamma$ range from~1 to~$d_i$, $d_j$, and
$d_k$ respectively. Clearly $[ijk]$ are symmetric in all three indices by the bi-invariance
of
$\langle\boldsymbol{\cdot}\,,\boldsymbol{\cdot}\rangle$.
Here we also have
$[ijk]=0$
if two indices coincide. Let $A:=[123]$.
It easy to see that $d_i\ge 2A$ for any $i=1,2,3$ (see \cite{Nikonorov2}), hence $(a_1,a_2,a_3)\in [0,1/2]^3$ where
$a_i=\frac{A}{d_i}$ for $i\in \{1,2,3\}$, see Table 1 for the value of these numbers for specific generalized Wallach spaces.
\smallskip

Note that these constants completely determine some important properties of a generalized Wallach space $G/H$, e.g. the equation
of the Ricci flow on~$G/H$, see \cite{AANS1,AANS2}.
Note also that Einstein metrics $(x_1,x_2,x_3)$ on a given generalized Wallach space are exactly the solutions
of the following polynomial system:
\begin{equation}\label{two_equat_sing1}
\begin{array}{l}
(a_2+a_3)(a_1x_2^2+a_1x_3^2-x_2x_3)+(a_2x_2+a_3x_3)x_1-(a_1a_2+a_1a_3+2a_2a_3)x_1^2=0,
\\
(a_1+a_3)(a_2x_1^2+a_2x_3^2-x_1x_3)+(a_1x_1+a_3x_3)x_2-(a_1a_2+2a_1a_3+a_2a_3)x_2^2=0.\\
\end{array}
\end{equation}
Indeed, the Ricci curvature of the Riemannian metric corresponding to the inner product $(\boldsymbol{\cdot}\,,\boldsymbol{\cdot})$ is
$\Ric=r_1 (\boldsymbol{\cdot}\,,\boldsymbol{\cdot})|_{\mathfrak{p}_1}+
r_2 (\boldsymbol{\cdot}\,,\boldsymbol{\cdot})|_{\mathfrak{p}_2}+r_3 (\boldsymbol{\cdot}\,,\boldsymbol{\cdot})|_{\mathfrak{p}_3}$, where
the principal Ricci curvatures $r_i$ satisfy the equation
$$
r_i=\frac{1}{2x_i}
+\frac{a_i}{2}
\biggl(
\frac{x_i}{x_j x_k}-
\frac{x_k}{x_i x_j}-
\frac{x_j}{x_i x_k}
\biggr),
$$
where
$i,j,k \in \{1,2,3\}$
and
$i \neq j \neq k \neq i$ (see Lemma 2 in \cite{Nikonorov2}).
The metric under consideration is Einstein if and only if $r_1=r_2=r_3$, that is equivalent to (\ref{two_equat_sing1}).
More details could be found in \cite{Lomshakov2} or  in \cite[Section 2]{AANS1}.

\begin{center}
\begin{figure}[t]
\centering\scalebox{1}[1]{
\includegraphics[angle=-90,totalheight=2.5in]{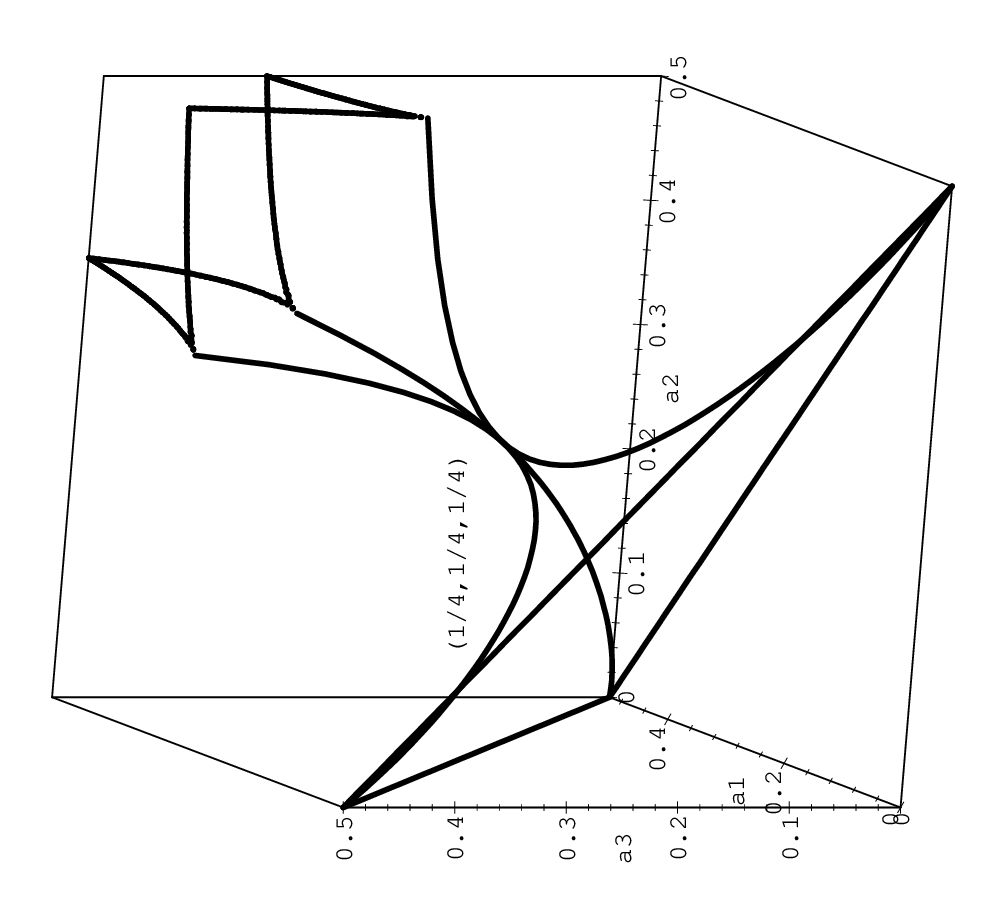}
\includegraphics[angle=-90,totalheight=2.5in]{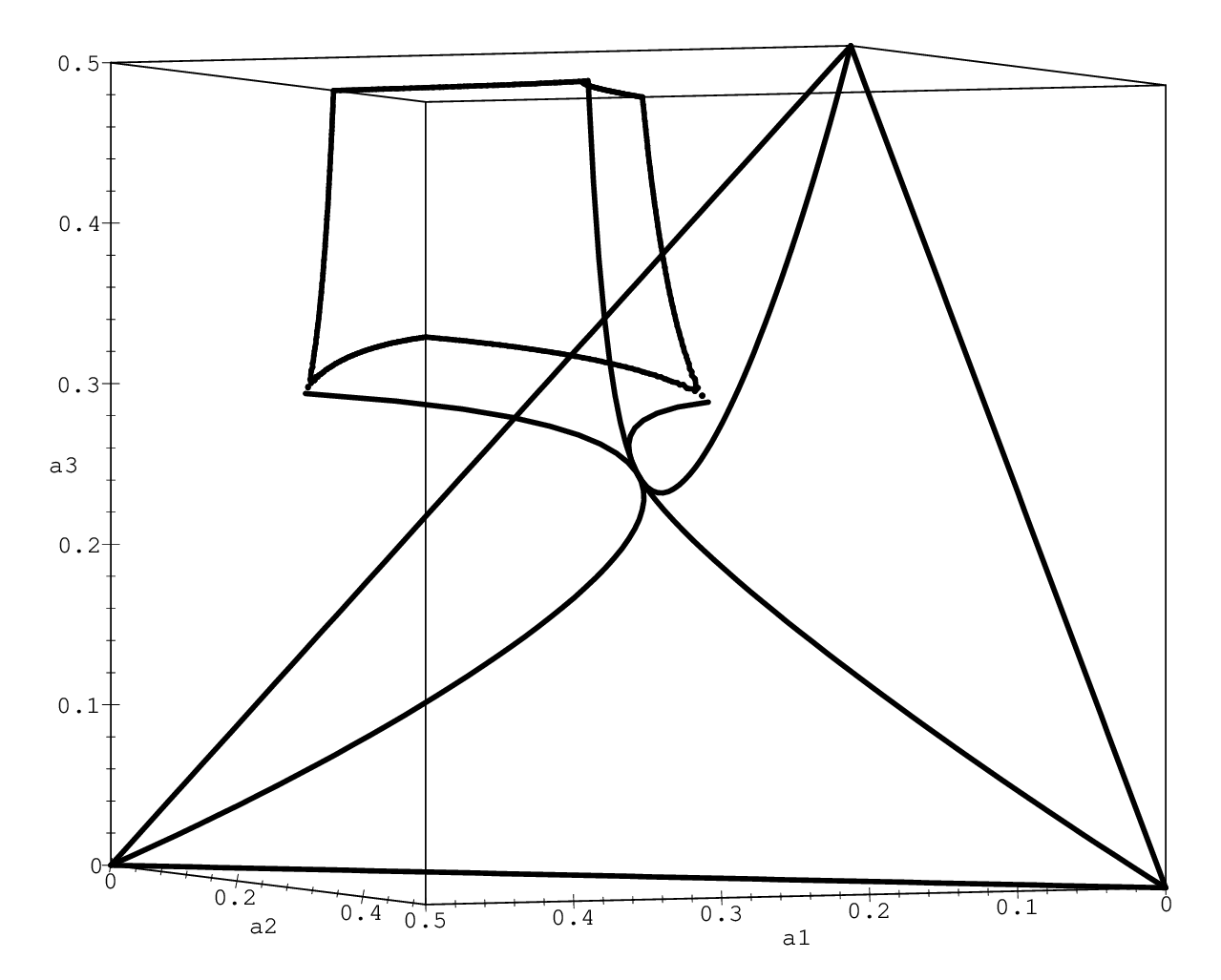}}
\caption{The surface $\Omega \cap [0,1/2]^3$.}
\label{singsur}
\end{figure}
\end{center}

Let $\Omega$ denote the algebraic surface in $\mathbb{R}^3$ defined by the equation $Q(a_1,a_2,a_3)=0$, where
\begin{eqnarray}\label{singval2}\notag
Q(a_1,a_2,a_3)\,=\,
(2s_1+4s_3-1)(64s_1^5-64s_1^4+8s_1^3+12s_1^2-6s_1+1\\\notag
+240s_3s_1^2-240s_3s_1-1536s_3^2s_1-4096s_3^3+60s_3+768s_3^2)\\
-8s_1(2s_1+4s_3-1)(2s_1-32s_3-1)(10s_1+32s_3-5)s_2\\\notag
-16s_1^2(13-52s_1+640s_3s_1+1024s_3^2-320s_3+52s_1^2)s_2^2\\\notag
+64(2s_1-1)(2s_1-32s_3-1)s_2^3+2048s_1(2s_1-1)s_2^4
\end{eqnarray}
and $s_1 = a_1+a_2+a_3$, $s_2 = a_1a_2+a_1a_3+a_2a_3$, $s_3 = a_1a_2a_3$.
Obviously, $Q(a_1,a_2,a_3)$ is a symmetric polynomial in $a_1,a_2,a_3$ of degree 12.
\smallskip

We recall some important properties of $\Omega$, see \cite{AANS1} for details.
The points $(0,0,1/2)$, $(0,1/2,0)$, and $(1/2,0,0)$ are all vertices of the cube $[0,1/2]^3$, which are also points of $\Omega$.
For $a_1=1/2$ and $a_2,a_3 \in (0,1/2]$, points of $\Omega$ form
a curve homeomorphic to the interval $[0,1]$ with endpoints $(1/2,1/2,\sqrt{2}/4\approx 0.3535533905)$ and $(1/2,\sqrt{2}/4\approx 0.3535533905,1/2)$
and with the singular point (a cusp) at the point $a_3=a_2=(\sqrt{5}-1)/4\approx 0.3090169942$.
The same is also valid under the permutation $a_1\to a_2 \to a_3\to a_1$.
\smallskip

The plain $s_1=a_1+a_2+a_3=1/2$ intersects
the set $\Omega \cap [0,1/2]^3$ exactly for points in the boundary of the triangle with the vertices
$(0,0,1/2)$, $(0,1/2,0)$, and $(1/2,0,0)$. For all other points in $\Omega \cap (0,1/2]^3$ we have the inequality
$s_1=a_1+a_2+a_3>1/2$.
\smallskip

Note that $(1/4,1/4,1/4)$ is the only point in $\Omega \cap [0,1/2]^3$
satisfying $s_1=a_1+a_2+a_3=3/4$. It turns out that the point $(1/4,1/4,1/4)$ is a singular point of degree $3$ of the algebraic surface
$\Omega$. This point is an elliptic umbilic (in the sense of Darboux) on the surface $\Omega$.
\smallskip

Now, we discuss a part of the surface $\Omega$ in the cube $(0,1/2)^3$.
Obviously, $\Omega$ is invariant under the permutation $a_1\to a_2\to a_3\to a_1$. It should be noted that the set
$(0,1/2)^3\cap \Omega$ is connected.
There are three curves (so-called ``edges'') of {\it singular points} on $\Omega$ (i.e. points where $\nabla Q =0$):
one of them has a parametric representation $a_1=-\frac{1}{2}\frac{16t^3-4t+1}{8t^2-1}, a_2=a_3=t$, and the others
are defined by permutations of $a_i$. The point $(1/4,1/4,1/4)$ is a common point of these three curves. The part of $\Omega$ in $(0,1/2)^3$
consists of three (pairwise isometric) ``bubbles'' spanned on every pair of ``edges''.
Note also that in~\cite{Ba}, the author found an explicit parameterization of the surface $\Omega$.
Another important fact is that the set $(0,1/2)^3\setminus \Omega$ has exactly three connected components.
More precisely, we have

\begin{theorem}[Theorem 1 in \cite{Abiev}]\label{odomains}
The following assertions hold with respect to the standard topology of $\mathbb{R}^3$:
\begin{enumerate}
  \item The set $(0,1/2)^3\cap \Omega$ is connected.
  \item The set $(0,1/2)^3\setminus \Omega$ consists of three conneted components.
\end{enumerate}
\end{theorem}

Another proof of these theorem could be found in \cite{BaBru}.
According to \cite{AANS1}, we denote by $O_1$, $O_2$, and $O_3$ the components containing the points
$(1/6,1/6,1/6)$, $(7/15,7/15,7/15)$, and $(1/6, 1/4, 1/3)$ respectively. Note that $Q(a_1,a_2,a_3)<0$ for $(a_1,a_2,a_3) \in O_1 \cup O_2$ and
$Q(a_1,a_2,a_3)>0$ for $(a_1,a_2,a_3) \in O_3$.
\smallskip

It is shown in \cite{AANS1}, that the normalized Ricci flow for a generalized Wallach space with $(a_1,a_2,a_3)\in (0,1/2)^3\setminus \Omega$
has no degenerate singular point, as a planar dynamical system.
The surface $\Omega$ was very important for the statement of
Theorem 6 in \cite{AANS1} (see also Theorem 7 in \cite{AANS2}), which
provides a general result about the type of the non-degenerate
singular points of the normalized Ricci flow for a generalized Wallach space
with given $a_1$, $a_2$, and $a_3$.
We recall it, taking in mind that singular points are exactly invariant Einstein metrics of fixed volume.

\begin{theorem}[Theorem 6 in \cite{AANS1}]\label{odomains}
Assume that $G/H$ is a generalized Wallach space such that $(a_1,a_2,a_3)\in O_i$.
\begin{enumerate}
  \item There are four singular points for $i=1$ including one unstable node and three saddles;
  \item There are four singular points for $i=2$ including one stable node and three saddles;
  \item There are two singular points for $i=3$ which are saddles.
\end{enumerate}
\end{theorem}

Now we describe the location of points $(a_1,a_2,a_3)\in \mathbb{R}^3$ determined by generalized Wallach spaces in Theorem \ref{classgws}.
Table 2 is to describe the region for every generalized Wallach space in Table 1.

{\small
\renewcommand{\arraystretch}{1.4}
\begin{table}[h]
{\bf Table 2.} The region for $G/H$ in Table 1.
\begin{center}
\begin{tabular}{|c|c|c|c|c|c|c|c|c|c|}
\hline
G/H & $(a_1,a_2,a_3)$ & G/H & $(a_1,a_2,a_3)$  & G/H & $(a_1,a_2,a_3)$  &  G/H & $(a_1,a_2,a_3)$  &  G/H & $(a_1,a_2,a_3)$ \\
\hline\hline
1&$O_1 \cup O_3 \cup \Omega$& 2 & $O_1$ & 3 & $O_1$ & 4 & $O_1$  &  5 & $O_3$ \\ \hline
6 & $O_3$ & 7 & $O_1$ & 8 & $O_3$ & 9 & $O_1$& 10 & $O_3$  \\ \hline
11 & $O_2$  &  12 & $O_3$ & 13 & $O_2$& 14 & $O_3$ & 15 & $O_1$ \\ \hline
\end{tabular}
\end{center}
\end{table}
\renewcommand{\arraystretch}{2}
}

In this table, the i-th generalized Wallach space means the one corresponding to the i-th pair in Table 1. Recall that every such space
determines the point $(a_1,a_2,a_3)$ and the points obtained with permutations of $a_1$, $a_2$, and $a_3$.
\smallskip

The corresponding result for the spaces $SO(k+l+m)/SO(k)\times SO(l)\times SO(m)$, where at least two of the numbers $k$, $l$, and $m$
are greater than $1$, was obtained in \cite{Nikonorov4}. If at least two of the numbers $k$, $l$, and $m$
are equal to $1$, then the point $(a_1,a_2,a_3)$ belongs to the boundary of the cube $(0,1/2)^3$.
We will discuss it in details in the next section.
It should be noted that these spaces are excepted from all other types of generalized Wallach spaces,
because the corresponding triples $(a_1,a_2,a_3)$ do not belong to only one region $O_i$.
\smallskip

For $SU(k+l+m)/S\big(U(k)\times U(l) \times U(m)\big)$, $k \geq l\geq m \geq 1$, we have
$$
a_1=\frac{k}{2(k+l+m)}, \quad a_2=\frac{l}{2(k+l+m)}, \quad a_3=\frac{m}{2(k+l+m)},
$$
and $a_1+a_2+a_3=1/2$.
It is clear that $(a_1,a_2,a_3)$ belongs to $O_1$. Moreover, the closure of the set of all such points coincides with the triangle in $\mathbb{R}^3$ with the
vertices $(0,0,1/2)$, $(0,1/2,0)$, and $(1/2,0,0)$. The last assertion easily follows from considering the barycentric coordinates in this triangle.
For these spaces we have simple explicit expressions for Einstein metrics, see \cite{Ar,Kim,Lomshakov2}.
\smallskip

For $Sp(k+l+m)/Sp(k)\times Sp(l) \times Sp(m)$, $k \geq l\geq m \geq 1$, we have
$$
a_1=\frac{k}{2(k+l+m+1)}, \quad a_2=\frac{l}{2(k+l+m+1)}, \quad a_3=\frac{m}{2(k+l+m+1)},
$$
and $a_1+a_2+a_3<1/2$. Hence $(a_1,a_2,a_3)$ also belongs to $O_1$.
Note that the original proof  in~\cite{Lomshakov2} of the fact that the spaces $Sp(k+l+m)/Sp(k)\times Sp(l) \times Sp(m)$
admits exactly four Einstein invariant metric (up to a homothety)
is very complicated and is based on the calculation of the Sturm's sequence of polynomials.
\smallskip

For the others, we may use one more simple observation: For $a_1=a_2=a_3=:a$, the point
$(a_1,a_2,a_3)$ is in $O_1$ (respectively, $O_2$), if $a<1/4$ (respectively, $a>1/4$).
Direct computations show that $(a_1,a_2,a_3)\in O_1$ for the spaces corresponding to {\bf lines 4, 7, 9},
and {\bf 15} of Table 1, $(a_1,a_2,a_3)\in O_3$ for the spaces corresponding to {\bf lines 5, 6, 8, 10, 12}, and {\bf 14} of Table 1,
and $(a_1,a_2,a_3)\in O_2$ for the spaces corresponding to the lines {\bf 11} and {\bf 13} of Table 1.
\smallskip

Theorem \ref{odomains} and Table 2 give us the explicit number of Einstein invariant metrics (up to a homothety)
on every generalized Wallach space $G/H$ with simple $G$ except $SO(k+l+m)/SO(k)\times SO(l)\times SO(m)$.

\begin{remark}
Note that the above examples satisfying $(a_1,a_2,a_3)\in O_2$ are interesting,
since they give an affirmative answer to the question of Prof. Christoph B\"{o}hm on the existence
of specific examples of generalized Wallach spaces with this property.
\end{remark}

If $G/H$ is a symmetric space which is a product of three irreducible symmetric spaces, then we get $A=0$ and $(a_1,a_2,a_3)=(0,0,0)$.
Note also that the equality $a_1=a_2=a_3=1/4$ holds for every space $(F \times F\times F \times F)/\diag(F)$,
as well as for the space $SO(6)/SO(2)^3$ (see details in \cite{Nikonorov4}).
Recall that the point $(1/4,1/4,1/4)$ is an elliptic umbilic on the surface $\Omega$.

\begin{remark}
The isotropy representation for the space $F^4/\diag(F)$, that is a Ledger--Obata space {\rm(}see \cite{CNN, LedgerObata, Nikonorov3}{\rm)},
can be decomposed into the sum of $3$ pairwise isomorphic $\Ad(\diag(F))$-modules which may not coincides with $\mathfrak{p}_i$,
$i=1,2,3$, as in Theorem \ref{classgws}, see \cite{Nikonorov3}.
The structure of invariant metrics on $F^k/\diag(F)$ was described in \cite{Nikonorov3}. Note that for $k=2$ we get irreducible symmetric spaces.
For $k=3$ and $k=4$, the classification of invariant Einstein metrics on the spaces $F^k/\diag(F)$
was obtained in \cite{Nikonorov3} and \cite{CNN} respectively,
but for $k \geq 5$ this problem is open.
\end{remark}

\section{The spaces $SO(k+l+m)/SO(k)\times SO(l)\times SO(m)$}

For the space $SO(k+l+m)/SO(k)\times SO(l) \times SO(m)$, $k \geq l\geq m \geq 1$, we have
$$
a_1=\frac{k}{2(k+l+m-2)}, \quad a_2=\frac{l}{2(k+l+m-2)}, \quad a_3=\frac{m}{2(k+l+m-2)}.
$$
If $l=m=1$, then we get $a_1=1/2$ and $a_2=a_3=\frac{1}{2k}$.
Hence, $(a_1,a_2,a_3) \not\in (0,1/2)^3$, but $(a_1,a_2,a_3)\in (0,1/2]^3$.
It is easy to check with using of (\ref{two_equat_sing1}) that such space has exactly one Einstein invariant metrics up to a homothety:
$(x_1,x_2,x_3)=(k+1,k+1,2k)$.
\smallskip

In what follows, we assume that $l\geq 2$. Therefore, $k\geq l \geq 2$ and $k+l+m \geq 5$.
It is proved in \cite[Section 6]{Nikonorov4} that $(a_1,a_2,a_3)\in O_1\cup O_3 \cup \Omega$ in this case.
Moreover, $(a_1,a_2,a_3)$ belongs to $O_1$, $O_3$, or $\Omega$, if
$Q(a_1,a_2,a_3)<0$, $Q(a_1,a_2,a_3)>0$, or $Q(a_1,a_2,a_3)=0$ respectively.
In the following, we will give some details.
\smallskip

Note that $a_1+a_2+a_3 = \frac{k+l+m}{2(k+l+m-2)}$. Since the function $x \mapsto \frac{x}{2(x-2)}=:g(x)$ decreases for $x>2$,
we know that $a_1+a_2+a_3 \leq 3/4=g(6)$ holds for any $k,l,m$ satisfying $k+m+l \geq 6$.
\smallskip

For $k+m+l \leq 5$ we should check only the space $SO(5)/SO(2)\times SO(2)\times SO(1)$ with $a_1=a_2=1/3$ and $a_3=1/6$.
It is easy to see that the point $(1/3,1/3,1/6)$ is in $O_3$, because the inequalities $a_i \geq 1/4$, $i=1,2,3$ hold for any points in $O_2$.
\smallskip

If  $k+m+l =6$, then we have two spaces $SO(6)/SO(2)^3$ and $SO(6)/SO(3)\times SO(2)\times SO(1)$ with $a_1+a_2+a_3 =3/4$.
Recall that the plane $a_1+a_2+a_3 =3/4$ intersects the surface $\Omega \cap (0,1/2)^3$ exactly in the point $(1/4,1/4,1/4)$
corresponding to the first space $SO(6)/SO(2)^3$. The space $SO(6)/SO(2)^3$ admits only standard Einstein metric,
which is highly degenerate as a singular point of the normalized Ricci flow, see details in \cite{AANS1} and \cite{AANS2}.
For the space $SO(6)/SO(3)\times SO(2)\times SO(1)$ we get $(a_1,a_2,a_3)=(3/8,1/4,1/8)\in O_3$.
\smallskip

If $k+m+l \geq 7$, we know $a_1+a_2+a_3 \leq 7/10=g(7)<3/4$. It follows that $(a_1,a_2,a_3)\not\in O_2$. For small values of $k\geq l\geq m$, we have the following Table.
{\small
\renewcommand{\arraystretch}{1.4}
\begin{table}[h]
{\bf Table 3.} The region for small $(k,l,m)$.
\begin{center}
\begin{tabular}{|c|c|c|c|c|c|c|c|c|c|}
\hline
$(k,l,m)$ & Region & $(k,l,m)$ & Region &$(k,l,m)$ & Region &$(k,l,m)$ & Region &$(k,l,m)$ & Region  \\
\hline\hline
(4,2,1) & $O_3$ & (3,3,1) & $O_3$ & (3,2,2) & $O_3$  &  (5,2,1) & $O_3$ & (4,3,1) & $O_3$ \\ \hline
(4,2,2) & $O_3$ & (3,3,2) & $O_3$ & (6,2,1) & $O_3$  & (5,3,1) & $O_3$ & (4,4,1) & $O_3$  \\ \hline
(5,2,2) & $O_3$ & (4,3,2) & $O_3$ &  (3,3,3) & $O_1$ & (7,2,1) & $O_3$ &  (6,3,1) & $O_3$ \\ \hline
(5,4,1) & $O_3$ & (6,2,2) & $O_3$ & (5,3,2) & $O_3$ &  (4,4,2) & $O_3$ & (4,3,3) & $O_3$ \\ \hline
(8,2,1) & $O_3$ & (7,3,1) & $O_3$ &  (6,4,1) & $O_3$ & (5,5,1) & $O_3$ & (7,2,2) & $O_3$ \\ \hline
(6,3,2) & $O_3$ &  (5,4,2) & $O_3$ & (5,3,3) & $O_3$ & (4,4,3) & $O_3$ & (9,2,1) & $O_3$ \\ \hline
(8,3,1) & $O_3$ & (7,4,1) & $O_3$ & (6,5,1) & $O_3$ & (8,2,2) & $O_3$ &  (7,3,2) & $O_3$ \\ \hline
(6,4,2) & $O_3$ & (5,5,2) & $O_3$ & (6,3,3) & $O_3$ &  (5,4,3) & $O_3$ & (4,4,4) & $O_1$ \\ \hline
(5,4,4) & $O_1$ & (6,4,4) & $O_1$ &  (5,5,4) & $O_3$ & (6,5,5) & $O_1$ & (7,6,5) & $O_1$ \\ \hline
\end{tabular}
\end{center}
\end{table}
\renewcommand{\arraystretch}{2}
}

Define a polynomial $G(k,l,m)$ in $k,l,m$ by
\begin{equation}\label{equag1}
\frac{G(k,l,m)}{2^{12}(k+l+m-2)^{12}}=Q\left(\frac{k}{2(k+l+m-2)},\frac{l}{2(k+l+m-2)},\frac{m}{2(k+l+m-2)}\right).
\end{equation}
It is clear that $G(k,l,m)$ is a symmetric (with respect to $k$, $l$ and $m$) polynomial of degree 12. Let $t_1 = k+l+m$, $t_2 = kl+km+lm$, and $t_3 = klm$.
Then we have \begin{equation}\label{equag2}
G(k,l,m)=H(t_1,t_2,t_3),
\end{equation}
where
\begin{eqnarray}\label{equag3}\notag
H(t_1,t_2,t_3)\,=\,
\big(16(t_1-2)^2+4t_3\big)\Big(1024t_1^5(t_1-2)^4-2048t_1^4(t_1-2)^5+512t_1^3(t_1-2)^6\\ \notag
+1536t_1^2(t_1-2)^7-1536t_1(t_1-2)^8+512(t_1-2)^9+3840t_3t_1^2(t_1-2)^4-4096t_3^3\\ \notag
-7680t_3t_1(t_1-2)^5-6144t_3^2t_1(t_1-2)^2+3840t_3(t_1-2)^6+6144t_3^2(t_1-2)^3\Big)\\
-8t_1\big(16(t_1-2)^2+4t_3\big)\big(16(t_1-2)^2-32t_3\big)\big(80(t_1-2)^2+32t_3\big)t_2\\ \notag
-16t_1^2\Big(832(t_1-2)^6-1664t_1(t_1-2)^5+2560t_3t_1(t_1-2)^2\\ \notag
+1024t_3^2-2560t_3(t_1-2)^3+832t_1^2(t_1-2)^4\Big)t_2^2\\ \notag
+1024(t_1-2)^2\big(16(t_1-2)^2-32t_3\big)t_2^3+32768t_1t_2^4(t_1-2)^2.
\end{eqnarray}
We see that the above Diophantine equation $G(k,l,m)=0$ is very complicated. Clearly,
\smallskip

\begin{lem}
The point $(a_1,a_2,a_3)=\left(\frac{k}{2(k+l+m-2)}, \frac{l}{2(k+l+m-2)}, \frac{m}{2(k+l+m-2)}\right)$ is in
$O_1$, $O_3$, or $\Omega$, if
$G(k,l,m)<0$, $G(k,l,m)>0$, or $G(k,l,m)=0$ respectively.
\end{lem}

It follows that

\begin{pred}\label{numberem}
If $G(k,l,m)<0$ {\rm(}respectively, $>0${\rm)}, then $SO(k+l+m)/SO(k)\times SO(l) \times SO(m)$
admits four {\rm(}respectively, two{\rm)} invariant Einstein metrics.
\end{pred}

\section{A further study on $SO(k+l+m)/SO(k)\times SO(l)\times SO(m)$}

In this section, we give a further study on $SO(k+l+m)/SO(k)\times SO(l)\times SO(m)$ satisfying $k\geq l \geq m$ and $l\geq 2$.
In particular we will prove Theorem \ref{ems} and Theorem \ref{emsn}. For this goal we get some more detailed information
on some special sets in $(0,1/2)^3$ related to the normalized Ricci flows on the spaces $SO(k+l+m)/SO(k)\times SO(l)\times SO(m)$.
In particular we prove

\begin{theorem}\label{emso}
Let $A$ be one of the set $\Omega$, $O_1$, and $O_3$.
Then there are infinitely many triples $(k,l,m)$ such that
$(a_1,a_2,a_3)=\left(\frac{k}{2(k+l+m-2)},\frac{l}{2(k+l+m-2)},\frac{m}{2(k+l+m-2)}\right)\in A$.
\end{theorem}

\begin{center}
\begin{figure}[t]
\centering\scalebox{1}[1]{
\includegraphics[angle=0,totalheight=2in]{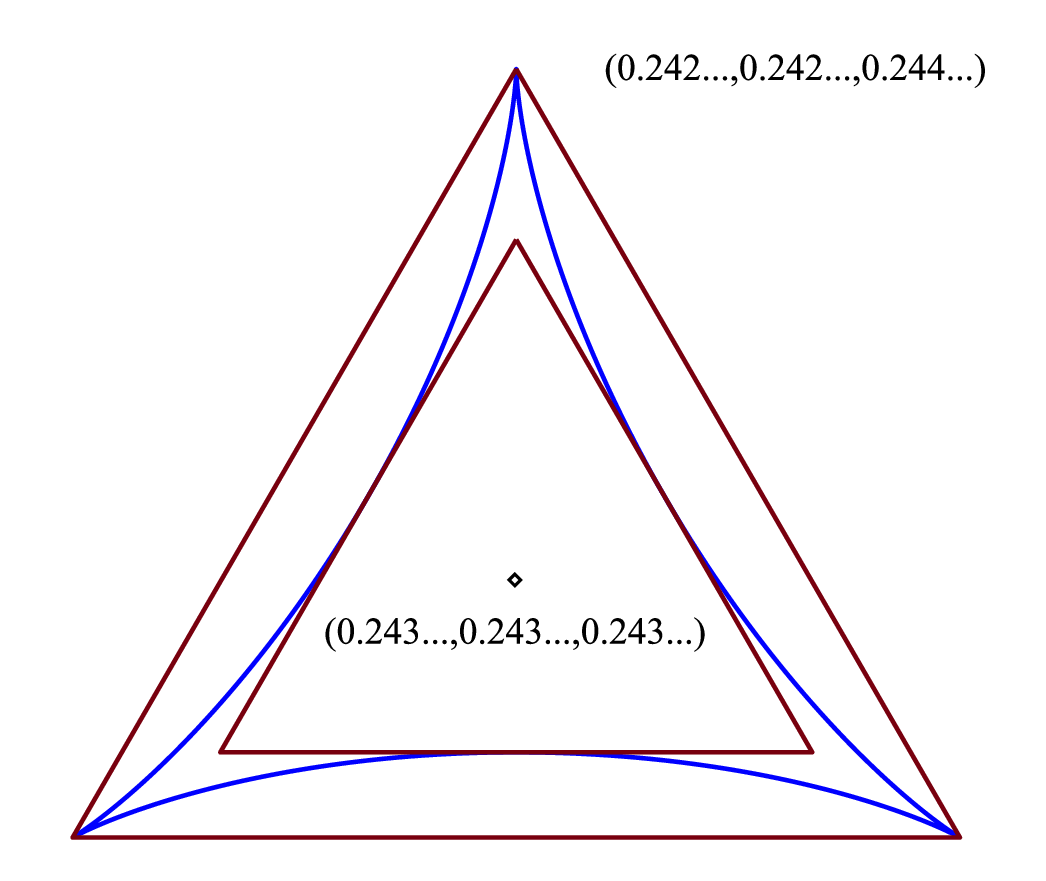}\qquad
\includegraphics[angle=0,totalheight=2in]{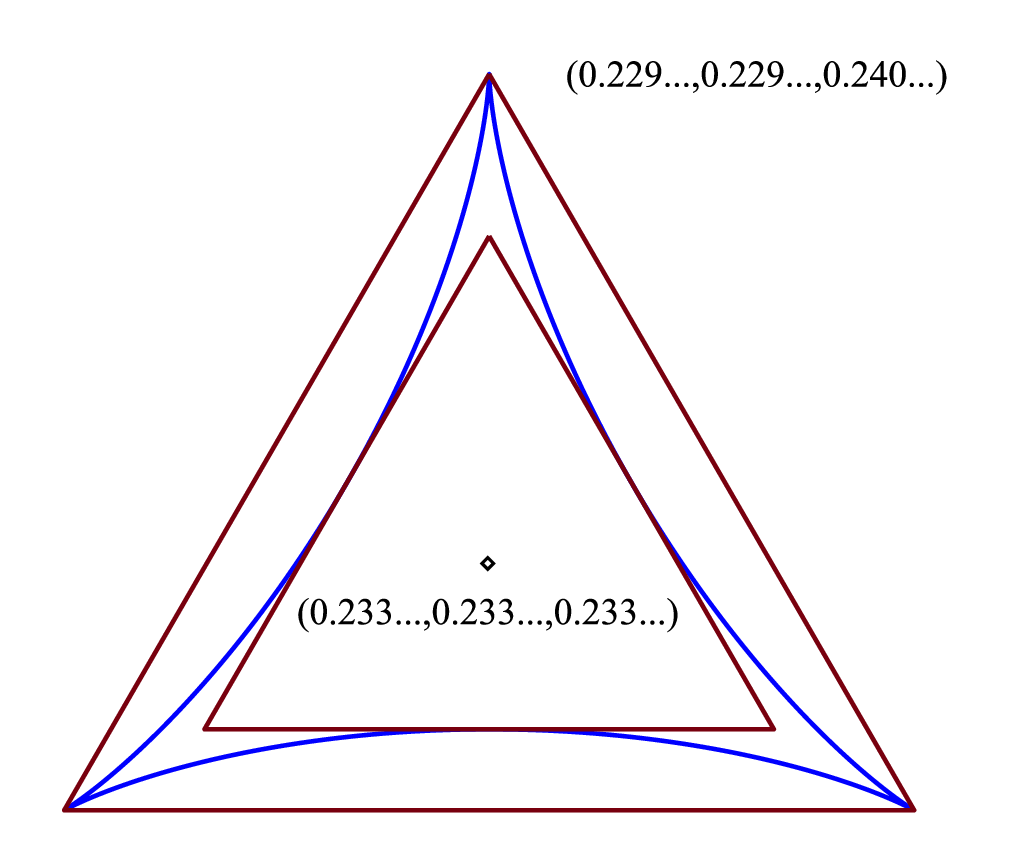}
}
\caption{The triangles $I(h)$, $IT(h)$, $ST(h)$ for $h=\frac{73}{100}$ and $h=\frac{7}{10}$.}
\label{curv1}
\end{figure}
\end{center}

In this section, we will assume that $k+m+l \geq 7$
(the cases $k+m+l < 7$ are discussed in the previous section).
\medskip

Let us fix some $h\in (1/2,3/4)$.
If $a_1+a_2+a_3=h$ and there exist $i,j \in \{1,2,3\}$ with $a_i=a_j$, then the equation $Q(a_1,a_2,a_3)=0$ is
\begin{equation}\label{ih1}
(4x-1-2h)(16x^2-(24h-4)x+8h^2-4h+1)(16x^3-16hx^2+2h-1)^3=0,
\end{equation}
where $a_i=a_j=x$, $a_k=h-2x$, $\{i,j,k\}=\{1,2,3\}$.
This equation has only two (distinct) roots in the interval $[0,h/2]$. One root is
\begin{equation}\label{ih2}
\widehat{x}=\widehat{x}(h)=\frac{6h-1-\sqrt{4h^2+4h-3}}{8}
\end{equation}
and the other one (that has multiplicity 3) is $\widetilde{x}=\widetilde{x}(h)$. Here, $\widetilde{x}$ is a root of the equation
\begin{equation}\label{ih3}
f(x):=16x^3-16hx^2+2h-1=0,
\end{equation}
that is in $(0,h/3)$. It should be noted that
the discriminant of the polynomial $f(x)$ is $D:=256(2h-1)(64h^3-54h+27)$ and $D>0$ for $h\in (1/2,3/4)$.
Hence the equation (\ref{ih3}) has three real roots. The first one is in $(-\infty,0)$, the second one is in $(0,h/3)$, and the third one is in $(h/2,\infty)$.
\smallskip
%Put e.~g. $(i,j)=(1,2)$. It is easy to see that $\widehat{x} \leq h/2$ for $h \geq 1/2$.

Note that $f(x)$ decreases for $x\in [0,2h/3]$.
It is easy to prove that $f(\widehat{x})<0=f(\widetilde{x})$, hence $\widetilde{x}<\widehat{x}$.
\smallskip

We need the following property of the surface $\Omega$ and the function $Q(a_1,a_2,a_3)$ (see (\ref{singval2})).

\begin{lem}\label{extrval1}
Suppose that a point $(a_1,a_2,a_3)\in (0,1/2)^3$ is such that $Q(a_1,a_2,a_3)=0$ and \linebreak
$\frac{\partial Q}{\partial a_i}(a_1,a_2,a_3)= \frac{\partial Q}{\partial a_j}(a_1,a_2,a_3)$, $i\neq j$, $i,j\in \{1,2,3\}$. Then
$(a_1-a_2)(a_1-a_3)(a_2-a_3)=0$.
\end{lem}

Note that this lemma imply (in particular) that all singular points $(a_1,a_2,a_3)$ of the surface
$\Omega=\{(a_1,a_2,a_3)\in \mathbb{R}^3\,|\, Q(a_1,a_2,a_3)=0\}$
(i.~e. points where $\nabla Q  =0$)
in $(0,1/2)^3$ are such that $(a_1-a_2)(a_1-a_3)(a_2-a_3)=0$. The proof of Lemma \ref{extrval1} is quite technical and could be found in Appendix.
\smallskip

For $h\in [1/2,3/4]$, we denote by $I(h)$ the intersection of the surface $\Omega \cap [0,1/2]^3$ with the plane $\pi(h):=\{(a_1,a_2,a_3)\,|\,a_1+a_2+a_3=h\}$.
Clear that $I(3/4)$ is the point $(1/4,1/4,1,4)$ and $I(1/2)$ is the boundary of the triangle with the vertices
$(0,0,1/2)$, $(0,1/2,0)$ and $(1/2,0,0)$.
 $I(h)$.
For $h\in [1/2,3/4)$, $I(h)$ contains exactly one point on every singular ``edge'' of $\Omega$
($I(3/4)$ contains one common point of all ``edges''). We call these special three points
(namely, $(\widetilde{x},\widetilde{x},h-2\widetilde{x})$, $(\widetilde{x},h-2\widetilde{x},\widetilde{x})$, and
$(h-2\widetilde{x},\widetilde{x},\widetilde{x})$)
as vertices of $I(h)$, that is a ``curvilinear triangle''.
Obviously, $I(h)$ is invariant under the group of symmetry of a usual triangle with the same vertices.
Clear that $I(h)$ is compact and it is contained in $(0,1/2)^3$ for all $h\in (1/2,3/4]$.
We are not going to emphasize the geometry of $I(h)$ in details, because all we need for our purposes is Lemma \ref{triangleincl} below.
\smallskip

Denote by $ST(h)$ the usual triangle (in the plane $\pi(h)$) with the same vertices as $I(h)$ has.
Now, let $IT(h)$ be a maximal usual triangle in $I(h)$ (in the plane $\pi(h)$, in particular) with sides parallel to sides of the triangle $ST(h)$
(it has the same symmetry group as $I(h)$ has).
\smallskip

Clearly, $ST(1/2)=IT(1/2)=I(1/2)$, $ST(3/4)=IT(3/4)=I(3/4)=\{(1/4,1/4,1/4)\}$, and it could be demonstrated that
$IT(h)\subsetneq I(h)\subsetneq ST(h)$ for $h\in(1/2,3/4)$
(see Figure~\ref{singsur}). In order to show the evolutions of these types of triangles, we reproduced the pictures in the planes $\pi(h)$
with four values of the parameter $h$ (see Figure~\ref{curv1} and Figure~\ref{curv2}).
\smallskip

\begin{lem}\label{triangleincl}
For any $h\in (1/2,3/4)$, $(\widehat{x},\widehat{x},h-2\widehat{x}),
(\widehat{x},h-2\widehat{x},\widehat{x}),(h-2\widehat{x},\widehat{x},\widehat{x}) \in IT(h)\cap I(h)$ and $I(h)\subset ST(h)$.
\end{lem}

\begin{proof} Recall that $\Omega=\{(a_1,a_2,a_3)\in \mathbb{R}^3\,|\, Q(a_1,a_2,a_3)=0\}$. The intersection of $\Omega \cap [0,1/2]^3$ with the plane
$\pi(h)=\{(a_1,a_2,a_3)\,|\,a_1+a_2+a_3=h\}$ is compact and is contained in $(0,1/2)^3$.
Clear that $(h/3,h/3,h/3)\not\in \Omega$, therefore,  $IT(h)$ is non-degenerate.
Clear that $IT(h)\cap I(h)$ is non-empty. If a point $(a_1,a_2,a_3)\in IT(h)\cap I(h)$, then either $(a_1,a_2,a_3)$ is a vertex of $IT(h)$
(and, consequently, of $I(h))$
or the gradient $\nabla Q (a_1,a_2,a_3)$ is orthogonal to one of the vectors $(1,-1,0)$, $(1,0,-1)$, $(0,1,-1)$
(because the sides of $IT(h)$ are parallel to these vectors). In the second case two components of $\nabla Q (a_1,a_2,a_3)$ are coincide and
we have $(a_1-a_2)(a_1-a_3)(a_2-a_3)=0$ by Lemma~\ref{extrval1}. The same we have in the first case,
since the vertices are of singular points on $\Omega$ (i.e. points where $\nabla Q =0$).
By the above discussion on  $\widehat{x}(h)$ and $\widetilde{x}(h)$ (see (\ref{ih2}) and (\ref{ih3})), we see that
(up to a permutation of coordinates) either $(a_1,a_2,a_3)=(\widehat{x},\widehat{x},h-2\widehat{x})$ or
$(a_1,a_2,a_3)=(\widetilde{x},\widetilde{x},h-2\widetilde{x})$.
Clear that the triangle with the vertices $(\widehat{x},\widehat{x},h-2\widehat{x})$,
$(\widehat{x},h-2\widehat{x},\widehat{x})$,  and $(h-2\widehat{x},\widehat{x},\widehat{x})$ are in the interior of the triangle $ST(h)$.
Therefore, $(a_1,a_2,a_3)=(\widehat{x},\widehat{x},h-2\widehat{x})$, and we prove the the first assertion of the lemma.

For the second assertion, suppose that $I(h)\not\subset ST(h)$. In the plane $\pi(h)$, consider a minimal triangle $MT(h)$
with sides parallel to the sides of $ST(h)$ and $I(h)\subset MT(h)$. Clearly $ST(h)$ is in the interior of $MT(h)$.
There is a point $(a_1,a_2,a_3)\in I(h)$ in the interior of some side of $MT(h)$.
Hence, two components of $\nabla Q (a_1,a_2,a_3)$ are coincide and
we have $(a_1-a_2)(a_1-a_3)(a_2-a_3)=0$ by Lemma~\ref{extrval1}.
By the above discussion on  $\widehat{x}(h)$ and $\widetilde{x}(h)$, we see that
(up to a permutation of coordinates) either $(a_1,a_2,a_3)=(\widehat{x},\widehat{x},h-2\widehat{x})$ or
$(a_1,a_2,a_3)=(\widetilde{x},\widetilde{x},h-2\widetilde{x})$. But both these points are in $ST(h)$.
This contradiction proves the second assertion of the lemma.
\end{proof}

\smallskip

Lemma \ref{triangleincl} and simple arguments about convex hulls of triples of points imply

\begin{lem}\label{lemma0}
Let $(a_1,a_2,a_3)\in (0,1/2)^3$ satisfying $a_1+a_2+a_3=h\in (1/2,3/4)$.
\begin{enumerate}
\item If $a_i> h-2\widehat{x}(h)$ for any $i=1,2,3$, then $(a_1,a_2,a_3)\in IT(h) \subset O_1$.
\item If there exists some $i\in \{1,2,3\}$ such that $a_i < \widetilde{x}(h)$,
then $(a_1,a_2,a_3)\in (0,1/2)^3 \cap \big(\pi(h)\setminus ST(h)\big) \subset O_3$.
\end{enumerate}
\end{lem}

\begin{center}
\begin{figure}[t]
\centering\scalebox{1}[1]{
\includegraphics[angle=0,totalheight=2in]{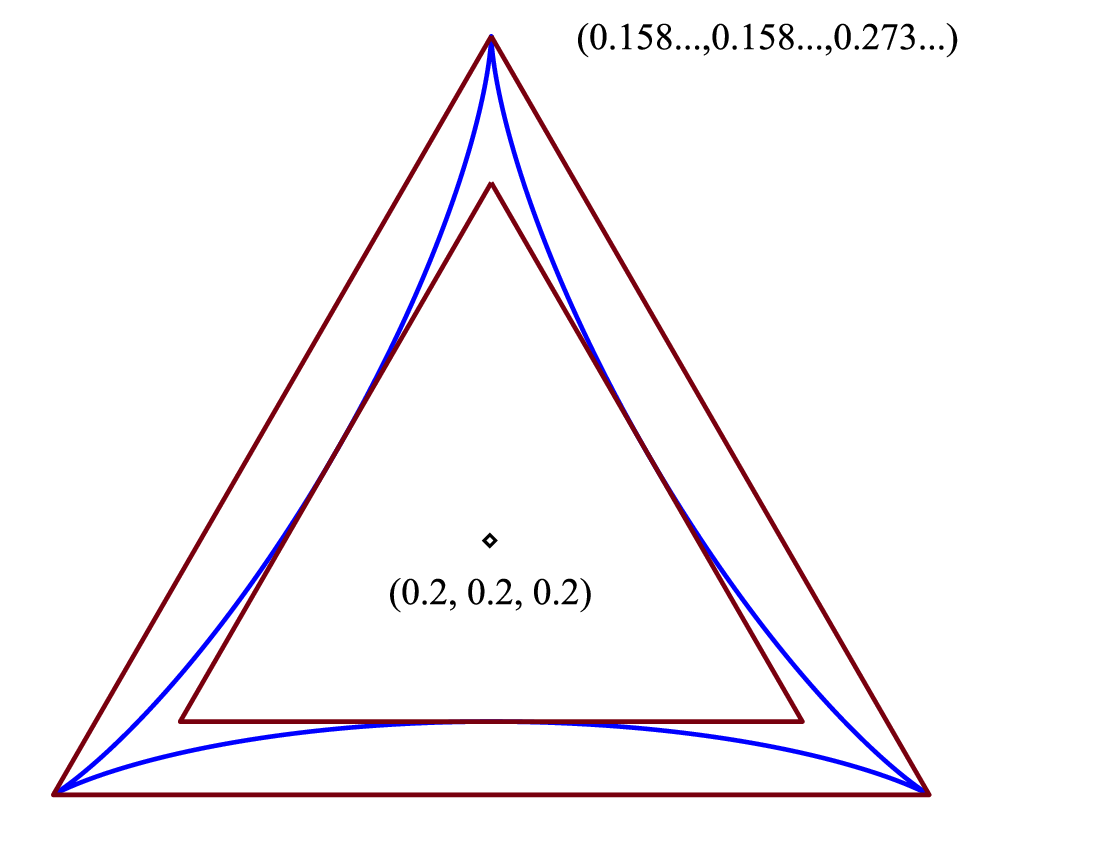}\qquad \quad
\includegraphics[angle=0,totalheight=2in]{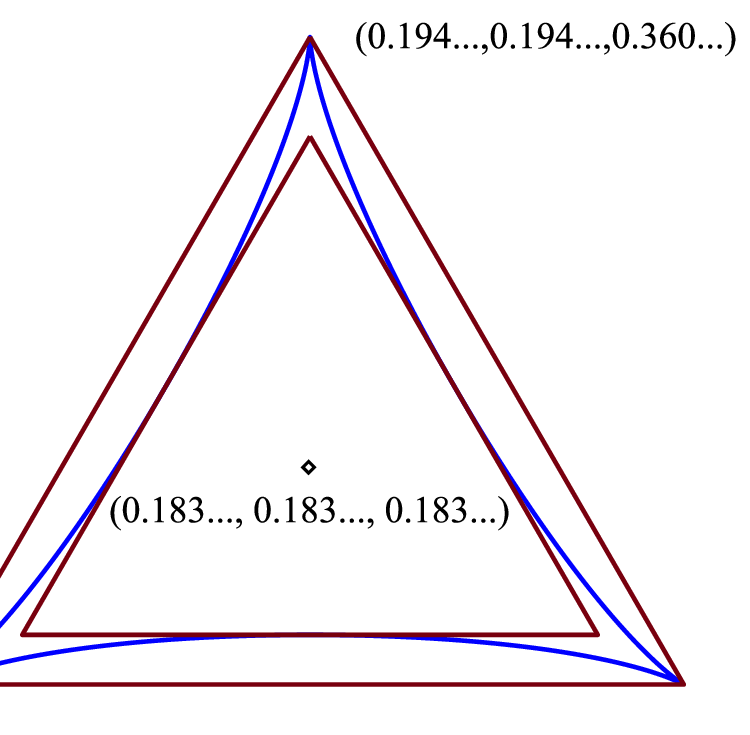}}
\caption{The triangles $I(h)$, $IT(h)$, $ST(h)$ for $h=\frac{3}{50}$ and $h=\frac{11}{20}$.}
\label{curv2}
\end{figure}
\end{center}

In what follows, we use the quantity
\begin{equation}\label{h0}
h_0:=\frac{k+l+m}{2(k+l+m-2)}.
\end{equation}
Clear, $h_0=a_1+a_2+a_3$ for $(a_1,a_2,a_3)=\left(\frac{k}{2(k+l+m-2)},\frac{l}{2(k+l+m-2)},\frac{m}{2(k+l+m-2)} \right)$.
Recall that $1/2<h_0 \leq 7/10 <3/4$ for $k+m+l \geq 7$.

\begin{pred}\label{prop1}
If $k\geq l\geq m$, $l\geq 2$, $k+m+l \geq 7$, and
\begin{equation}\label{ih5}
\frac{k+l}{4(k+l+m-2)}< \widehat{x}(h_0),
\end{equation}
then $(a_1,a_2,a_3) \in O_1$.
\end{pred}

\begin{proof}
Note that the inequality (\ref{ih5}) is equivalent to $a_3=\frac{m}{2(k+l+m-2)}> h_0 -2\widehat{x}(h_0)$. It implies that
$(a_1,a_2,a_3) \in IT(h) \subset O_1$ by $a_1\geq a_2 \geq a_3$ and Lemma \ref{lemma0}.
\end{proof}
\smallskip

Let $h^{\ast}=\frac{3\sqrt{2}-2}{4}$. Clearly the function $h \mapsto \widehat{x}(h)$ decreases on the interval $[1/2,h^{\ast}]$
and increases on the interval $[h^{\ast},3/4]$. Then we have
$\widehat{x}(h)\geq \frac{\sqrt{2}-1}{2}$ since $\widehat{x}(h^{\ast})=\frac{\sqrt{2}-1}{2}\approx 0,2071067810<\widehat{x}(1/2)=\widehat{x}(3/4)=1/4$.
By Proposition~\ref{prop1}, we have the following corollary.

\begin{cor}\label{cor1}
If $k\geq l \geq m$, $l\geq 2$, $k+m+l \geq 7$, and
\begin{equation}\label{ih6}
\frac{k+l}{4(k+l+m-2)}<\frac{\sqrt{2}-1}{2},
\end{equation}
then $(a_1,a_2,a_3) \in O_1$.
\end{cor}

\begin{pred}\label{prop1n}
If $k\geq l \geq m$, $l\geq 2$, $k+m+l \geq 7$,  and
\begin{equation*}
m> \sqrt{2k+2l-4},
\end{equation*}
then $(a_1,a_2,a_3) \in O_1$.
\end{pred}

\begin{proof}
Since the function $x \mapsto 16x^2-(24h-4)x+8h^2-4h+1$ decreases on $[0,h/2]$ for $h\in [1/2,3/4]$ (since $24h-4\geq 16h$),
we know that the inequality (\ref{ih5}) is equivalent to
$$
16\left(\frac{k+l}{4(k+l+m-2)}\right)^2-(24h_0-4)\left(\frac{k+l}{4(k+l+m-2)}\right)+8h_0^2-4h_0+1>0,
$$
which is equivalent to $m^2-2k-2l+4>0$. Equivalently, $m> \sqrt{2k+2l-4}$. Now, it suffices to apply Proposition \ref{prop1}.
\end{proof}

\begin{remark}
Note that $\frac{k+l}{4(k+l+m-2)}\leq \widehat{x}(h_0)$ implies $(a_1,a_2,a_3) \in O_1 \cup \Omega$.
Here $(a_1,a_2,a_3) \in \Omega$, if and only if $l=k$, $a_3=\frac{m}{2(2k+m-2)}=h_0-2\widehat{x}(h_0)$,
$a_1=a_2=\frac{k}{2(2k+m-2)}=\widehat{x}(h_0)$. This property is fulfilled if and only if $(k,l,m)=(t^2+1,t^2+1,2t)$, where $3\leq t \in \mathbb{Z}$.
Note that $m=\sqrt{2k+2l-4}$ for this case.
\end{remark}

\begin{pred}\label{prop2}
If $k\geq l \geq m$, $l\geq 2$, $k+m+l \geq 7$,  and \begin{equation}\label{ih8}
\frac{m}{2(k+l+m-2)}< \widetilde{x}(h_0),
\end{equation}
then $(a_1,a_2,a_3) \in O_3$.
\end{pred}

\begin{proof}
Note that the inequality (\ref{ih8}) is equivalent to $a_3=\frac{m}{2(k+l+m-2)}<\widetilde{x}(h_0)$. It implies that
$(a_1,a_2,a_3) \in  O_3$ by Lemma \ref{lemma0}.
\end{proof}

\begin{remark}\label{remark3}
Note that $\frac{m}{2(k+l+m-2)}\leq \widetilde{x}(h_0)$ implies $(a_1,a_2,a_3) \in O_3 \cup \Omega$,
and $(a_1,a_2,a_3) \in \Omega$ if and only if $(a_1,a_2,a_3)$ is a vertex of $I(h_0)$.
\end{remark}

\begin{pred}\label{prop3}
The point $(a_1,a_2,a_3)=\left(\frac{k}{2(k+l+m-2)},\frac{l}{2(k+l+m-2)},\frac{m}{2(k+l+m-2)}\right)$ is a vertex of $I(h_0)$ if and only if $k=l=m=2$.
\end{pred}

\begin{proof} Obviously, if $k=l=m=2$, the point $(a_1,a_2,a_3)$ is a vertex of $I(3/4)$ (because it is an unique point in this degenerate triangle). On the other hand, assume that $k\geq l\geq m$ and $(a_1,a_2,a_3)$ is a vertex of $I(h_0)$. Clearly $(a_1,a_2,a_3)=(h_0-2\widetilde{x},\widetilde{x},\widetilde{x})$. Hence, $m=l$ and
$k^2-(l-2)^2k-l^3+4l^2-8l+4=0$. It is easy to see that $k=l=2$ since the discriminant is $l^2((l-2)^2+4)$.
\end{proof}

\begin{pred}\label{prop4}
If $k\geq l \geq m$, $l\geq 2$, $k+m+l \geq 7$,  and
\begin{equation*}\label{ih9}
m<\Bigl(1+\sqrt{k+l}\Bigr)\frac{k+l-2}{k+l-1},
\end{equation*}
then $(a_1,a_2,a_3) \in O_3$.
\end{pred}

\begin{proof}
Recall that $f(x)$ decreases for $x\in [0,2h/3]$. Since $\frac{m}{2(k+l+m-2)}\leq \frac{1}{3}\,h_0=\frac{k+l+m}{6(k+l+m-2)}$ and
$\widetilde{x}(h_0) \leq h_0/3$, the inequality (\ref{ih8}) is equivalent to $f\left(\frac{m}{2(k+l+m-2)}\right)>0$, i.e.
$$(k+l-1)m^2-2(k+l-2)m-k^2-2kl-l^2+4k+4l-4<0.$$ That is, $m<\Bigl(1+\sqrt{k+l}\Bigr)\frac{k+l-2}{k+l-1}.$
\end{proof}
\smallskip

Note that the function $\eta(x)=\bigl(1+\sqrt{x}\bigr)\frac{x-2}{x-1}$ increases for $x>1$ and
$\eta(3)\approx1.366$, $\eta(4)=2$, $\eta(5)\approx2.427$, $\eta(6)\approx2.759$, $\eta(7)\approx3.038$.
This implies

\begin{cor}\label{cor3}
If $k\geq l \geq m$, $l\geq 2$, and $k+m+l \geq 7$,  then $(a_1,a_2,a_3) \in O_3$ if
\begin{enumerate}
\item  $m=1$;
\item  $m=2$ and $k+l\geq 5$;
\item  $m=3$ and $k+l \geq 7$.
\end{enumerate}
\end{cor}

Since $\sqrt{x}\leq \bigl(1+\sqrt{x}\bigr)\frac{x-2}{x-1}$ for $x\geq 4$, we have the following corollary from Proposition \ref{prop4}.

\begin{cor}\label{cor4}
If $k\geq l \geq m$, $l\geq 2$, $k+m+l \geq 7$, and
\begin{equation*}\label{ih9}
m<\sqrt{k+l},
\end{equation*}
then $(a_1,a_2,a_3) \in O_3$.
\end{cor}

Now, we are ready to prove Theorem \ref{ems},  Theorem \ref{emsn}, and Theorem \ref{emso}.
\smallskip

\begin{proof}[Proof of Theorem \ref{ems}]
From Proposition \ref{prop1n} and Corollary \ref{cor4} we get that
$(a_1,a_2,a_3) \in O_1$ for $m> \sqrt{2k+2l-4}$ and $(a_1,a_2,a_3) \in O_3$ for $m<\sqrt{k+l}$ (all cases with $k+m+l <7$ are discussed in the previous section).
Now, it suffices to apply Theorem \ref{odomains},
taking in mind that singular points for the normalized Ricci flow are exactly invariant Einstein metrics of fixed volume.
\end{proof}
\smallskip

\begin{proof}[Proof of Theorem \ref{emso}]
For $(k,l,m)=(t^2+1,t^2+1,2t)$, where $t\in \mathbb{N}$, we get
\begin{eqnarray*}
(a_1,a_2,a_3)=\left(\frac{k}{2(k+l+m-2)},\frac{l}{2(k+l+m-2)},\frac{m}{2(k+l+m-2)}\right)=\\
\left(\frac{t^2+1}{4t(t+1)},\frac{t^2+1}{4t(t+1)},\frac{1}{2(t+1)}\right)\in \Omega
\end{eqnarray*}
by (\ref{singval2}) and direct calculations. Alternatively, it suffices to check that $G(t^2+1,t^2+1,2t)=0$, see~(\ref{equag2}).
Then Theorem~\ref{emso} follows from this observation, Proposition \ref{prop1n}, and Corollary \ref{cor4}.
\end{proof}

\begin{remark}\label{remim}
It is easy to check  using of {\rm(}\ref{two_equat_sing1}{\rm)} that $SO\bigl(2(t^2+t+1)\bigr)/ SO(t^2+1)\times SO(t^2+1)\times SO(2t)$,
the generalized Wallach space with $(a_1,a_2,a_3)=\left(\frac{t^2+1}{4t(t+1)},\frac{t^2+1}{4t(t+1)},\frac{1}{2(t+1)}\right)$,
has exactly three Einstein metrics {\rm(}up to constant multiple{\rm)} for $t \geq 2$, whereas $SO(6)/SO(2)\times SO(2)\times SO(2)$ has unique
{\rm(}up to a homothety{\rm)}
Einstein invariant metric. Indeed, for the case $a_1=a_2$, subtracting the second equation by the first equation in $\eqref{two_equat_sing1}$, we have
\begin{equation}\label{remark5-1}2a_1(a_1+2a_3)(x_1^2-x_2^2)-(a_1+2a_3)x_3(x_1-x_2)=0;\end{equation} adding the first equation into the second equation in $\eqref{two_equat_sing1}$, we have
\begin{equation}\label{remark5-2}-2a_1a_3(x_1^2+x_2^2)+2a_1(a_1+a_3)x_3^2+2a_1x_1x_2-a_1(x_1+x_2)x_3=0.\end{equation}
By the equation~$\eqref{remark5-1}$, we have $$x_1=x_2, \text{ or }  2a_1(x_1+x_2)=x_3.$$
Putting $x_1=x_2$ and $(a_1,a_2,a_3)=\left(\frac{t^2+1}{4t(t+1)},\frac{t^2+1}{4t(t+1)},\frac{1}{2(t+1)}\right)$ into the equation~$\eqref{remark5-2}$, we have $$(2tx_1-(t+1)x_3)^2=0.$$ That is, we have the solution $$(x_1,x_2,x_3)=\Bigl( t+1,t+1,2t \Bigr)$$ up to a homothety. Putting $2a_1(x_1+x_2)=x_3$ and  $(a_1,a_2,a_3)=\left(\frac{t^2+1}{4t(t+1)},\frac{t^2+1}{4t(t+1)},\frac{1}{2(t+1)}\right)$ into the equation~$\eqref{remark5-2}$, we have the following solutions $$
\left(2t^2+(t^2-1)\sqrt{\frac{t}{t+2}}\,, 2t^2-(t^2-1)\sqrt{\frac{t}{t+2}}\,, \frac{2t(t^2+1)}{t+1} \right),
$$
$$
\left(2t^2-(t^2-1)\sqrt{\frac{t}{t+2}}\,, 2t^2+(t^2-1)\sqrt{\frac{t}{t+2}}\,, \frac{2t(t^2+1)}{t+1} \right).
$$
See details in \cite{Lomshakov2}. It is clear that the second and the third metrics are isometric {\rm(}in fact, one need only to permute $x_1$ and $x_2${\rm)}.
Note that $\lim\limits_{t \to \infty}\frac{x_3}{x_1}$ is equal to $2$, $2/3$, and $2$ for these three families, whereas
$\lim\limits_{t \to \infty}\frac{x_3}{x_2}$ is equal to $2$, $2$, and $2/3$ respectively.
\end{remark}

\begin{proof}[Proof of Theorem \ref{emsn}]
The proof for $q=2$ and $q=4$ follows directly from Theorem \ref{ems}.
The proof for $q=3$ follows from Remark \ref{remim}: every space
$SO\bigl(2(t^2+t+1)\bigr)/ SO(t^2+1)\times SO(t^2+1)\times SO(2t)$ with $t\geq 2$ has exactly three invariant Einstein metrics up to a homothety.
\end{proof}
\smallskip

Finally, we propose the following interesting open questions.

\begin{ques} Is there a triple $(k,l,m)$ with $G(k,l,m)=0$ such that the corresponding space $SO(k+l+m)/SO(k)\times SO(l)\times SO(m)$ admits
even numbers of Einstein metrics?
\end{ques}

\begin{ques} Is $(k,l,m)=(2,2,2)$ a unique triple with $G(k,l,m)=0$ and $k\geq l \geq 2$,
such that the corresponding space $SO(k+l+m)/SO(k)\times SO(l)\times SO(m)$ admits
exactly one Einstein metric?
\end{ques}

\section*{Appendix}

Here we consider a detailed proof of Lemma \ref{extrval1}. All computations below were produced
using Maple (of course, it is possible to apply any other system of symbolic computations).
\smallskip

Without loss of generality, we assume that $i=1$ and $j=2$.
Suppose that the lemma is not true, i.~e. $(a_1-a_2)(a_1-a_3)(a_2-a_3)\neq 0$.

Direct calculations (see see (\ref{singval2})) show that
$\frac{\partial Q}{\partial a_1}- \frac{\partial Q}{\partial a_2}=(a_2-a_1)\left(\frac{\partial Q}{\partial s_2}+
a_3\frac{\partial Q}{\partial s_3}\right)=8(a_1-a_2)\cdot Q_{12}$, where

\begin{eqnarray*}
Q_{12}= 8192a_1^5a_2^3a_3^2+12288a_1^5a_2^2a_3^3+4096a_1^5a_2a_3^4+16384a_1^4a_2^4a_3^2+36864a_1^4a_2^3a_3^3\\
+28672a_1^4a_2^2a_3^4+8192a_1^4a_2a_3^5+8192a_1^3a_2^5a_3^2+36864a_1^3a_2^4a_3^3+49152a_1^3a_2^3a_3^4\\
+24576a_1^3a_2^2a_3^5+4096a_1^3a_2a_3^6+12288a_1^2a_2^5a_3^3+28672a_1^2a_2^4a_3^4+24576a_1^2a_2^3a_3^5\\
+8192a_1^2a_2^2a_3^6+4096a_1a_2^5a_3^4+8192a_1a_2^4a_3^5+4096a_1a_2^3a_3^6-2048a_1^5a_2^3-2304a_1^5a_2^2a_3\\
-1024a_1^5a_2a_3^2-768a_1^5a_3^3-4096a_1^4a_2^4-8960a_1^4a_2^3a_3-9472a_1^4a_2^2a_3^2-3840a_1^4a_2a_3^3\\
+256a_1^4a_3^4-2048a_1^3a_2^5-8960a_1^3a_2^4a_3-16896a_1^3a_2^3a_3^2-10752a_1^3a_2^2a_3^3+1024a_1^3a_2a_3^4\\
+2304a_1^3a_3^5-2304a_1^2a_2^5a_3-9472a_1^2a_2^4a_3^2-10752a_1^2a_2^3a_3^3+1536a_1^2a_2^2a_3^4+5888a_1^2a_2a_3^5\\
+1280a_1^2a_3^6-1024a_1a_2^5a_3^2-3840a_1a_2^4a_3^3+1024a_1a_2^3a_3^4+5888a_1a_2^2a_3^5+2560a_1a_2a_3^6\\
-768a_2^5a_3^3+256a_2^4a_3^4+2304a_2^3a_3^5+1280a_2^2a_3^6+1024a_1^4a_2^3+1152a_1^4a_2^2a_3+512a_1^4a_2a_3^2\\
+384a_1^4a_3^3+1024a_1^3a_2^4+2304a_1^3a_2^3a_3+3072a_1^3a_2^2a_3^2+1024a_1^3a_2a_3^3-512a_1^3a_3^4\\
+1152a_1^2a_2^4a_3+3072a_1^2a_2^3a_3^2+1280a_1^2a_2^2a_3^3-1024a_1^2a_2a_3^4-640a_1^2a_3^5+512a_1a_2^4a_3^2\\
+1024a_1a_2^3a_3^3-1024a_1a_2^2a_3^4-1280a_1a_2a_3^5+384a_2^4a_3^3-512a_2^3a_3^4-640a_2^2a_3^5+208a_1^5a_2\\
+176a_1^5a_3+736a_1^4a_2^2+656a_1^4a_2a_3+144a_1^4a_3^2+1056a_1^3a_2^3+896a_1^3a_2^2a_3-624a_1^3a_2a_3^2\\
-560a_1^3a_3^3+736a_1^2a_2^4+896a_1^2a_2^3a_3-1632a_1^2a_2^2a_3^2-2576a_1^2a_2a_3^3-880a_1^2a_3^4+208a_1a_2^5\\
+656a_1a_2^4a_3-624a_1a_2^3a_3^2-2576a_1a_2^2a_3^3-1888a_1a_2a_3^4-384a_1a_3^5+176a_2^5a_3+144a_2^4a_3^2\\
-560a_2^3a_3^3-880a_2^2a_3^4-384a_2a_3^5-32a_3^6-208a_1^4a_2-176a_1^4a_3-528a_1^3a_2^2-272a_1^3a_2a_3\\
+32a_1^3a_3^2-528a_1^2a_2^3-96a_1^2a_2^2a_3+864a_1^2a_2a_3^2+528a_1^2a_3^3-208a_1a_2^4-272a_1a_2^3a_3\\
+864a_1a_2^2a_3^2+1184a_1a_2a_3^3+352a_1a_3^4-176a_2^4a_3+32a_2^3a_3^2+528a_2^2a_3^3+352a_2a_3^4+32a_3^5\\
+40a_1^4+212a_1^3a_2+148a_1^3a_3+320a_1^2a_2^2+284a_1^2a_2a_3+20a_1^2a_3^2+212a_1a_2^3+284a_1a_2^2a_3\\
+8a_1a_2a_3^2-88a_1a_3^3+40a_2^4+148a_2^3a_3+20a_2^2a_3^2-88a_2a_3^3-24a_3^4-60a_1^3-180a_1^2a_2\\
-96a_1^2a_3-180a_1a_2^2-192a_1a_2a_3-12a_1a_3^2-60a_2^3-96a_2^2a_3-12a_2a_3^2+24a_3^3+30a_1^2\\
+60a_1a_2+18a_1a_3+30a_2^2+18a_2a_3-12a_3^2-5a_1-5a_2+2a_3\\
\end{eqnarray*}

By the condition of the lemma we have $Q(a_1,a_2,a_3)=Q_{12}(a_1,a_2,a_3)=0$.
Let us consider the resultant $R_1$ of the polynomials $Q$ and $Q_{12}$ with respect to $a_1$.
By direct computations we have
\begin{eqnarray*}
R_1=R_1(t,s)=-17592186044416\,s^3t(2t-1)^2(2s-1)^4(2t-1+2s)(s-t)^3(t+s)^2\\
\times(16s^3+16s^2t-4s-2t+1)^3 (8s^2t+8st^2+8s^2+20st+8t^2+12s+12t+5)\\
\times(64s^2t+64st^2-32s^2-80st-32t^2+24s+24t-5)^2 \cdot L^3\,,
\end{eqnarray*}
where $t=a_2$, $s=a_3$ for simplicity, and
\begin{eqnarray*}
L=16384s^{10}t^4+32768s^9t^5-16384s^8t^6-65536s^7t^7-16384s^6t^8+32768s^5t^9\\
+16384s^4t^{10}-13312s^{10}t^2-30720s^9t^3-8192s^8t^4+30720s^7t^5+43008s^6t^6\\
+30720s^5t^7-8192s^4t^8-30720s^3t^9-13312s^2t^{10}+6656s^9t^2+15872s^8t^3\\
+2560s^7t^4-25088s^6t^5-25088s^5t^6+2560s^4t^7+15872s^3t^8+6656s^2t^9+3136s^{10}\\
+7936s^9t+6016s^8t^2-384s^7t^3-9152s^6t^4-15104s^5t^5-9152s^4t^6-384s^3t^7\\
+6016s^2t^8+7936st^9+3136t^{10}-3136s^9-7712s^8t-4768s^7t^2+4384s^6t^3\\
+11232s^5t^4+11232s^4t^5+4384s^3t^6-4768s^2t^7-7712st^8-3136t^9+368s^8\\
+272s^7t-192s^6t^2-272s^5t^3-352s^4t^4-272s^3t^5-192s^2t^6+272st^7+368t^8\\
+736s^7+1872s^6t+336s^5t^2-2944s^4t^3-2944s^3t^4+336s^2t^5+1872st^6+736t^7\\
-408s^6-616s^5t+520s^4t^2+1440s^3t^3+520s^2t^4-616st^5-408t^6+56s^5-80s^4t\\
-408s^3t^2-408s^2t^3-80st^4+56t^5+40s^4+164s^3t+240s^2t^2+164st^3+40t^4\\
-30s^3-78s^2t-78st^2-30t^3+7s^2+13st+7t^2.
\end{eqnarray*}

We should have $R_1(s,t)=0$. Note that $a_3=s=t=a_2$ is impossible.
It is easy to see that
\begin{eqnarray*}
s^3t(2t-1)^2(2s-1)^4(t+s)^2&>& 0, \\
8s^2t+8st^2+8s^2+20st+8t^2+12s+12t+5&> &0,\\
64s^2t+64st^2-32s^2-80st-32t^2+24s+24t-5&<&0
\end{eqnarray*}
for $(t,s)\in (0,1/2)\times (0,1/2)$. It is not difficult to show also that $L> 0$ for
$(t,s)\in (0,1/2)\times (0,1/2)$
(the minimal value of $L$ on $[0,1/2]\times [0,1/2]$ is $0$ and it is attained  only at the points $(0,0)$, $(0,1/2)$, $(1/2,0)$, and $(1/2,1/2)$).
Therefore, we get either $2t+2s=1$ or $16s^3+16s^2t-4s-2t+1=0$.

\smallskip

{\bf Suppose that $s=1/2-t$, $t\in (0,1/2)$.} Then
$Q=4a_1 \cdot Q_{c1}$, where

\begin{eqnarray*}
Q_{c1}=-1024a_1^5s^4+2048a_1^3s^6-1024a_1s^8+1024a_1^5s^3-512a_1^4s^4\\
-3072a_1^3s^5+2048a_1s^7+512s^8+320a_1^5s^2+512a_1^4s^3+1152a_1^3s^4\\
-1984a_1s^6-1024s^7-288a_1^5s-192a_1^4s^2+128a_1^3s^3+448a_1^2s^4+1184a_1s^5\\
+768s^6+44a_1^5+32a_1^4s-328a_1^3s^2-448a_1^2s^3-452a_1s^4-256s^5-4a_1^4\\
+116a_1^3s+196a_1^2s^2+108a_1s^3+32s^4-13a_1^3-42a_1^2s-13a_1s^2+4a_1^2.
\end{eqnarray*}

$Q_{12}=8a_1(2s+2a_1-1)\cdot Q_{12c1}$, where

\begin{eqnarray*}
Q_{12c1}=-1024a_1^4s^4+1024a_1^3s^5+1024a_1^2s^6-1024a_1s^7+1536a_1^4s^3-1792a_1^3s^4\\
-2304a_1^2s^5+1792a_1s^6+768s^7+128a_1^4s^2+640a_1^3s^3+1536a_1^2s^4-1152a_1s^5\\
-1664s^6-464a_1^4s-48a_1^3s^2-80a_1^2s^3+816a_1s^4+1184s^5+76a_1^4-36a_1^3s\\
-396a_1^2s^2-476a_1s^3-320s^4+4a_1^3+192a_1^2s+124a_1s^2+24s^3-23a_1^2-11a_1s.
\end{eqnarray*}

If $2s+2a_1-1=0$, then $a_1=t=a_2$, that is impossible.

Let us consider the resultant $R_2$ of the polynomials $Q_{c1}$ and $Q_{12c1}$ with respect to $a_1$.
Direct computations implies
\begin{eqnarray*}
R_2=R_2(s)=-2048\,s^6(8s^2-4s-13)(2s-1)^4(4s-1)^{10}\\
\times (4096s^6-6144s^5-768s^4+3328s^3-192s^2-384s+91)^3.
\end{eqnarray*}

It is easy to check that $8s^2-4s-13<0$ and $4096s^6-6144s^5-768s^4+3328s^3-192s^2-384s+91>0$ for $s\in (0,1/2)$.
If $s=1/4$, then $Q_{c1}=\frac{1}{128}(2a_1+1)(4a_1-1)^4$, hence, $a_1=1/4=s=a_3$. Therefore, we have no critical points with pairwise
distinct $a_1,a_2,a_3$ in this case.

\smallskip

{\bf Now, suppose that  $16s^3+16s^2t-4s-2t+1=16s^3-4s+1-2t(1-8s^2)=0$.} Then we get $t =\frac{1-4s+16s^3}{2(1-8s^2)}$,
where $s\in \left(0, \frac{\sqrt{5}-1}{4}\approx 0.3090169942\right)$ due to $t\in(0,1/2)$.
Under the above substitution we get that $Q=\frac{4(a_1-s)^2}{(8s^2-1)^6}\cdot Q_{c2}$, where

\begin{eqnarray*}
Q_{c2}= -33554432a_1^4s^{14}-33554432a_1^3s^{15}+33554432a_1^2s^{16}+33554432a_1s^{17}\\
+16777216a_1^4s^{13}+33554432a_1^3s^{14}-33554432a_1s^{16}-16777216s^{17}+41943040a_1^4s^{12}\\
+25165824a_1^3s^{13}-50331648a_1^2s^{14}-25165824a_1s^{15}+8388608s^{16}-26214400a_1^4s^{11}\\
-47185920a_1^3s^{12}+9437184a_1^2s^{13}+47185920a_1s^{14}+16777216s^{15}-14483456a_1^4s^{10}\\
+9568256a_1^3s^{11}+30867456a_1^2s^{12}-5505024a_1s^{13}-13369344s^{14}+12648448a_1^4s^9\\
+14876672a_1^3s^{10}-16580608a_1^2s^{11}-21430272a_1s^{12}-3670016s^{13}+835584a_1^4s^8\\
-8224768a_1^3s^9-2965504a_1^2s^{10}+11747328a_1s^{11}+6569984s^{12}-2506752a_1^4s^7\\
-278528a_1^3s^8+5488640a_1^2s^9+974848a_1s^{10}-1466368s^{11}+395264a_1^4s^6\\
+1318912a_1^3s^7-1906688a_1^2s^8-3170304a_1s^9-812032s^{10}+180224a_1^4s^5-370688a_1^3s^6\\
-126976a_1^2s^7+1191936a_1s^8+535552s^9-67072a_1^4s^4-2048a_1^3s^5+342784a_1^2s^6\\
+2816a_1s^7-117504s^8+2560a_1^4s^3+25856a_1^3s^4-113664a_1^2s^5-151168a_1s^6\\
+1408s^7+2608a_1^4s^2-7200a_1^3s^3+7152a_1^2s^4+54720a_1s^5+10112s^6-592a_1^4s\\
+816a_1^3s^2+5168a_1^2s^3-7600a_1s^4-6432s^5+44a_1^4-8a_1^3s-1672a_1^2s^2-440a_1s^3\\
+2508s^4-4a_1^3+228a_1^2s+364a_1s^2-588s^3-13a_1^2-62a_1s+75s^2+4a_1-4s\, .
\end{eqnarray*}

We also have $Q_{12}=\frac{8(a_1-s)(16a_1s^2+16s^3-2a_1-4s+1)}{(8s^2-1)^6}\cdot Q_{12c2}$, where

\begin{eqnarray*}
Q_{12c2}=
8388608a_1^4s^{13}-8388608a_1^3s^{14}-8388608a_1^2s^{15}+8388608a_1s^{16}-4194304a_1^4s^{12}\\
+8388608a_1^2s^{14}-4194304s^{16}-13107200a_1^4s^{11}+13631488a_1^3s^{12}+9961472a_1^2s^{13}\\
-11534336a_1s^{14}+1048576s^{15}+7864320a_1^4s^{10}-14680064a_1^2s^{12}+1048576a_1s^{13}\\
+5767168s^{14}+5799936a_1^4s^9-8912896a_1^3s^{10}-1277952a_1^2s^{11}+7274496a_1s^{12}\\
-2883584s^{13}-4472832a_1^4s^8+1343488a_1^3s^9+6832128a_1^2s^{10}-1376256a_1s^{11}\\
-2326528s^{12}-761856a_1^4s^7+2777088a_1^3s^8-1630208a_1^2s^9-1875968a_1s^{10}+1712128s^{11}\\
+1064960a_1^4s^6-983040a_1^3s^7-1155072a_1^2s^8+458752a_1s^9+172032s^{10}-61440a_1^4s^5\\
-282624a_1^3s^6+745472a_1^2s^7+223232a_1s^8-334336s^9-108032a_1^4s^4+210432a_1^3s^5\\
-70656a_1^2s^6-87040a_1s^7+34560s^8+20224a_1^4s^3-22144a_1^3s^4-91648a_1^2s^5+6400a_1s^6\\
+24960s^7+3200a_1^4s^2-10880a_1^3s^3+38720a_1^2s^4+3840a_1s^5-3776s^6-1128a_1^4s\\
+3520a_1^3s^2-2552a_1^2s^3-4112a_1s^4-1776s^5+76a_1^4-336a_1^3s-1620a_1^2s^2\\
+1736a_1s^3+576s^4+4a_1^3+370a_1^2s-320a_1s^2-54s^3-23a_1^2+22a_1s+s^2\, .
\end{eqnarray*}

The equality $16a_1s^2+16s^3-2a_1-4s+1=0$ implies $a_1=t=a_2$  and $a_1-s$ implies $a_1=s=a_3$, that is impossible.

Let us consider the resultant $R_3$ of the polynomials $Q_{c2}$ and $Q_{12c2}$ with respect to $a_1$.
Direct computations implies
\begin{eqnarray*}
R_3=R_3(s)=442368\,s^4(16s^3-4s+1)(64s^4-48s^3-80s^2+68s-13)(2s+1)^2\\
\times (4s^2+2s-1)^2(64s^4-32s^2+8s-1)^2(4s-1)^5(2s-1)^8(8s^2-1)^{12}\cdot M\,,
\end{eqnarray*}

where
\begin{eqnarray*}
M=67108864s^{16}-73400320s^{14}+6291456s^{13}+34799616s^{12}-6029312s^{11}\\
-8978432s^{10}+2510848s^9+1233920s^8-570368s^7-42496s^6+66496s^5\\
-11664s^4-1952s^3+1184s^2-204s+13.
\end{eqnarray*}

It is easy to see that $R_3(s)=0$ for $s\in \left(0, \frac{\sqrt{5}-1}{4}\right)$ if and only if $s=1/4$. But in this case we have $a_2=t=1/4=s=a_3$ that is
impossible. Lemma \ref{extrval1} is completely proved.

\bigskip

{\bf Acknowlegment.}
The authors are grateful to the referees for helpful comments and suggestions that improved the presentation of this paper.

\vspace{10mm}

\bibliographystyle{amsunsrt}

\vspace{5mm}

\end{document}